\documentclass[11pt]{amsart}
\usepackage{amsmath,amssymb,accents}
\usepackage{graphicx}

\usepackage{fullpage}
\usepackage{enumerate}
\usepackage{pgf,tikz}
\usetikzlibrary{arrows}
\usetikzlibrary{positioning}
\usetikzlibrary{patterns}

\def\COMMENT#1{}
\let\COMMENT=\footnote
\def\TASK#1{}

\newdimen\margin   % needed for macros \textdisplay & \ltextdisplay
\def\textno#1&#2\par{%
    \margin=\hsize
    \advance\margin by -4\parindent
           \setbox1=\hbox{\sl#1}%
    \ifdim\wd1 < \margin
       $$\box1\eqno#2$$%
    \else
       \bigbreak
       \hbox to \hsize{\indent$\vcenter{\advance\hsize by -3\parindent
       \sl\noindent#1}\hfil#2$}%
       \bigbreak
    \fi}
\makeatletter
\newcommand*{\overrightharpoonup}{\mathpalette{\overarrow@\rightharpoonupfill@}}
\newcommand*{\rightharpoonupfill@}{\arrowfill@\relbar\relbar\rightharpoonup}

\DeclareFontFamily{U}{matha}{\hyphenchar\font45}
\DeclareFontShape{U}{matha}{m}{n}{
      <5> <6> <7> <8> <9> <10> gen * matha
      <10.95> matha10 <12> <14.4> <17.28> <20.74> <24.88> matha12
      }{}
\DeclareSymbolFont{matha}{U}{matha}{m}{n}
\DeclareMathSymbol{\varrightharpoonup}{3}{matha}{"E1}

\newcommand{\diedge}[1]{\overrightharpoonup{#1}}

\usepackage{graphicx}

\newcommand{\sat}{\text{sat}}
\newcommand{\qedclaim}{
\hfill\scalebox{0.6}{$\blacksquare$}\\
}
\newcommand{\inv}[1]{\rotatebox[origin=c]{180}{#1}}
\let\leq\leqslant

\let\geq\geqslant
\let\ge\geqslant
\usepackage[colorlinks=true, linkcolor=blue, filecolor=magenta, urlcolor=cyan]{hyperref}

\usepackage{amsthm}
\usepackage{enumitem}
\def\alabel{\upshape({\itshape \!\alph*})}
\def\rmlabel{\upshape({\itshape \roman*\,})}

\def\newparagraph{\phantom{.}\medskip}

\usepackage{dsfont}

\newtheorem{theorem}{Theorem}[section]

\newtheorem{example}[theorem]{Example}
\newtheorem{lemma}[theorem]{Lemma}
\newtheorem{claim}[theorem]{Claim}

\newtheorem{conjecture}[theorem]{Conjecture}
\newtheorem{proposition}[theorem]{Proposition}

\newtheorem*{theorem*}{Theorem}
\newtheorem*{define*}{Definition}
\newtheorem*{example*}{Example}
\newtheorem*{lem*}{Lemma}
\newtheorem*{claim*}{Claim}
\newtheorem*{fact*}{Fact}
\newtheorem*{col*}{Corollary}
\newtheorem*{conj*}{Conjecture}

\let\svthefootnote\thefootnote

\newenvironment{proofclaim}{\removelastskip\penalty55\medskip\noindent{\bf Proof of the claim. }}

\title{The induced saturation problem for posets}

\author{Andrea Freschi,  Sim\'on Piga, Maryam Sharifzadeh and Andrew Treglown}

\thanks{AF: University of Birmingham, United Kingdom, {\tt axf079@bham.ac.uk}
\\ \indent
SP: University of Birmingham, United Kingdom, {\tt s.piga@bham.ac.uk}, research supported by EPSRC grant 
\indent EP/V002279/1. 
\\ \indent MS:  Ume{\aa} Universitet, Sweden, {\tt maryam.sharifzadeh@umu.se}
\\  \indent
AT: University of Birmingham, United Kingdom, {\tt a.c.treglown@bham.ac.uk}, research supported by EPSRC grant
\indent EP/V002279/1.}

\begin{document}
\maketitle
%\writedatatofile

\maketitle

% ABSTRACT
% CT papers must include an abstract. The abstract should consist of a
% succinct statement of background followed by a listing of the
% principal new results that are to be found in the paper. The abstract
% should be informative, clear, and as complete as possible. Phrases
% like "we investigate..." or "we study..." should be kept to a minimum
% in favor of "we prove that..."  or "we show that...".  Do not
% include equation numbers, unexpanded citations (such as "[23]"), or
% any other references to things in the paper that are not defined in
% the abstract. The abstract may be distributed without the rest of the
% paper so it must be entirely self-contained.  Try to include all words
% and phrases that someone might search for when looking for your paper.
% You can use some basic LaTeX commands in the abstract, but not any
% user defined macros. 

\begin{abstract}
For a fixed poset $P$, a family $\mathcal F$ of subsets  of $[n]$ is induced $P$-saturated if $\mathcal F$ does not contain an induced copy of $P$, but for every subset $S$ of $[n]$ such that $ S\not \in \mathcal F$, $P$ is an induced subposet of  $\mathcal F \cup \{S\}$.  The size of the smallest such family $\mathcal F$ is denoted by $\text{sat}^* (n,P)$.
Keszegh,  Lemons,  Martin, P\'alv\"olgyi and  Patk\'os [Journal of Combinatorial Theory Series A, 2021] proved that there is a dichotomy of behaviour for this parameter: given any poset $P$, either $\text{sat}^* (n,P)=O(1)$  or $\text{sat}^* (n,P)\geq \log _2 n$. In this paper we improve this general result showing that either  $\text{sat}^* (n,P)=O(1)$ or  $\text{sat}^* (n,P) \geq \min\{ 2 \sqrt{n}, n/2+1\}$. Our proof makes use of a Tur\'an-type result for digraphs.

Curiously, it remains open as to whether our result is  essentially best possible or not.
On the one hand, a conjecture of Ivan states that for the so-called diamond poset $\Diamond$ we have  $\text{sat}^* (n,\Diamond)=\Theta (\sqrt{n})$; so if true this conjecture implies our  
 result is tight up to a multiplicative constant.
On the other hand, a  conjecture of Keszegh,  Lemons,  Martin, P\'alv\"olgyi and  Patk\'os states that
given any poset $P$, either $\text{sat}^* (n,P)=O(1)$  or $\text{sat}^* (n,P)\geq n+1$.
We prove that this latter conjecture is true for a certain class of posets $P$.
\end{abstract}

% TABLE OF CONTENTS, LIST OF FIGURES, LIST OF TABLES
% Please, do not include a table of contents, a list of figures, or a
% list of tables. They will be removed by the editors (and the command
% is actually redefined in the ct.sty file).

%%%%%%%%%%%%%%%%%%%%%%%%%%%%%%%%%%%%%%%%%%%%%%%%%%%
%%%%%%%%%%%%%%%%%%%%%%%%%%%%%%%%%%%%%%%%%%%%%%%%%%%
\let\thefootnote\relax\footnote{The main results of this paper were first announced in the conference abstract~\cite{fpst}.}
\addtocounter{footnote}{-1}\let\thefootnote\svthefootnote

\section{Introduction}
\emph{Saturation} problems have been  well studied in graph theory. A graph $G$ is \emph{$H$-saturated} if it does not contain a  copy of the graph $H$, but adding any edge to $G$ from its complement creates a copy of $H$. Tur\'an's celebrated theorem~\cite{turan} can be stated in the language of saturation: it determines the maximum number of edges in a $K_r$-saturated $n$-vertex graph. In contrast, Erd\H{o}s, Hajnal and Moon~\cite{moon} determined the minimum
number of edges in a $K_r$-saturated $n$-vertex graph; see the survey~\cite{ejcsurvey} for further results in this direction.

In recent years there has been an emphasis on developing the theory of saturation for posets. Tur\'an-type problems have been extensively studied in this setting (see, e.g., the survey~\cite{griggs}). In this paper we are interested in \emph{minimum} saturation questions \`a la Erd\H{o}s--Hajnal--Moon. In particular, we consider \emph{induced saturation} problems.

All posets  we consider will be (implicitly) viewed as finite collections of finite subsets of $\mathds N$. In particular, we say that $P$ is a \emph{poset on $[p]:=\{1,2,\dots,p\}$} if $P$ consists of subsets of $[p]$.
Let $P,Q$ be posets. A \emph{poset homomorphism} from $P$ to $Q$ is a function $\phi: \ P \rightarrow Q$ such that for every $A,B\in P$, if $A \subseteq B$ then $\phi(A)\subseteq \phi(B)$.
 We say that $P$ is a \emph{subposet} of $Q$ if there is an injective poset homomorphism from $P$ to $Q$; otherwise, $Q$ is said to be \emph{$P$-free}. Further we say $P$ is an \emph{induced subposet} of $Q$ if there is an injective poset homomorphism $\phi$ from $P$ to $Q$ such that for every $A,B \in P$, $\phi(A)\subseteq \phi(B)$ if and only if $A\subseteq B$; otherwise, $Q$ is said to be \emph{induced $P$-free}.
  
  For a fixed poset $P$, we say that a family $\mathcal F\subseteq 2^{[n]}$ of subsets  of $[n]$ is \emph{$P$-saturated} if $\mathcal F$ is $P$-free, but for every subset $S$ of $[n]$ such that $ S\not \in \mathcal F$, then $P$ is a subposet of  $\mathcal F \cup \{S\}$. A family $\mathcal F\subseteq 2^{[n]}$ of subsets  of $[n]$ is \emph{induced $P$-saturated} if $\mathcal F$ is induced $P$-free, but for every subset $S$ of $[n]$ such that $ S\not \in \mathcal F$, then $P$ is an induced subposet of  $\mathcal F \cup \{S\}$. 
  
  The study of minimum saturated posets was initiated by Gerbner, Keszegh, Lemons, Palmer, P\'alv\"olgyi and Patk\'os~\cite{gklppp} in 2013. In their work the relevant parameter is
  $\text{sat} (n,P)$, which is defined to be the size of the smallest $P$-saturated family of subsets of $[n]$. See, e.g.,~\cite{gklppp, klmpp, morris} for various results on $\text{sat} (n,P)$.
  
  The induced analogue   of   $\text{sat} (n,P)$ -- denoted by $\text{sat}^* (n,P)$ -- was first considered by Ferrara, Kay, Kramer,  Martin, Reiniger, Smith and Sullivan~\cite{fkkm}. Thus,
  $\text{sat}^* (n,P)$ is defined to be the size of the smallest induced $P$-saturated family of subsets of $[n]$. 
  The following result from~\cite{klmpp} (and implicit in~\cite{fkkm}) shows that the parameter~$\sat^*(n,P)$ has a dichotomy of behaviour.

\begin{theorem}\cite{fkkm, klmpp} \label{bound} 
For any poset $P$, either there exists a constant $K_P$ with $\text{sat}^*(n, P) \leq K_P $ or $\text{sat}^*(n, P) \geq \log _2 n$, for all $n \in \mathds N$.
\end{theorem}
  
Probably the most important open problem in the area is to obtain a tight version of Theorem~\ref{bound}; that is, to replace the $\log _2 n$ in Theorem~\ref{bound} with a term  that is as large as possible.
In fact, Keszegh, Lemons, Martin, P\'alv\"olgyi and Patk\'os~\cite{klmpp} made the following conjecture in this direction.  
  
\begin{conjecture}\cite{klmpp} \label{conj1} 
For any poset $P$, either there exists a constant $K_P$ with $\text{sat}^*(n, P)\leq K_P$ or $\text{sat}^*(n, P) \geq  n+1$, for all $n \in \mathds N$.
\end{conjecture}
    Note that the lower bound of $n+1$ is rather natural here. For example,  it is the size of the largest \emph{chain} in $2^{[n]}$ as well as the smallest possible size of the union of two consecutive `layers' in $2^{[n]}$, namely the layer containing $[n]$ and the layer containing all subsets of $[n]$ of size exactly $n-1$.
    Furthermore, such structures form minimum induced saturated families for the so-called 
    fork poset $\vee$, i.e., $\text{sat}^*(n, \vee) = n+1$~\cite{fkkm}; so the lower bound in Conjecture~\ref{conj1} cannot be increased. There are also no known examples of posets $P$ for which
    $\text{sat}^*(n, P)=\omega (n)$.
    
 In contrast, Ivan~\cite[Section~3]{ivan1} presented evidence that led her to conjecture a rather different picture
 for the  \emph{diamond} poset  $\Diamond$ (see Figure~\ref{fig:posets} for the Hasse diagram of $\Diamond$).

      \begin{conjecture}\cite{ivan1} \label{ivanconj}  
$\text{sat}^*(n, \Diamond) = \Theta (\sqrt{n}).$
    \end{conjecture}

  Our main result is the following improvement of Theorem~\ref{bound}.
\begin{theorem}\label{mainthm} For any poset $P$, either there exists a constant $K_P$ with $\text{sat}^*(n, P) \leq K_P $ or  $\text{sat}^*(n, \mathcal P) \geq \min \{ 2 \sqrt{n}, n/2+1 \}$ for all   $n \in \mathds N$. 
\end{theorem}
Note that if $n \geq 12$ then  $n/2+1 \geq 2 \sqrt{n}$. Thus, if Conjecture~\ref{ivanconj} is true, the lower bound in Theorem~\ref{mainthm} would be tight up to a multiplicative constant. 

\begin{figure}[h]\label{fig:posets}
\hspace{0.1cm}\\
\centering
\begin{tikzpicture}
\draw[black] (0,0)--(0,2)--(2,0)--(2,2);
\fill[black] (0,0) circle (3pt);
\fill[black] (2,0) circle (3pt);
\fill[black] (0,2) circle (3pt);
\fill[black] (2,2) circle (3pt);

\draw[black] (5,0)--(5,1)--(4,2)--(5,1)--(6,2);
\fill[black] (5,0) circle (3pt);
\fill[black] (5,1) circle (3pt);
\fill[black] (4,2) circle (3pt);
\fill[black] (6,2) circle (3pt);

\draw[black] (9,0)--(8,1)--(9,2)--(10,1)--(9,0);
\fill[black] (9,0) circle (3pt);
\fill[black] (8,1) circle (3pt);
\fill[black] (10,1) circle (3pt);
\fill[black] (9,2) circle (3pt);

\draw[black] (12,0)--(13,1)--(12,2);
\draw[black] (14,0)--(13,1)--(14,2);
\fill[black] (12,0) circle (3pt);
\fill[black] (12,2) circle (3pt);
\fill[black] (13,1) circle (3pt);
\fill[black] (14,0) circle (3pt);
\fill[black] (14,2) circle (3pt);
\end{tikzpicture}
\caption{Hasse diagrams for the posets $N$, $Y$, $\Diamond$ and $X$.} 
\end{figure}
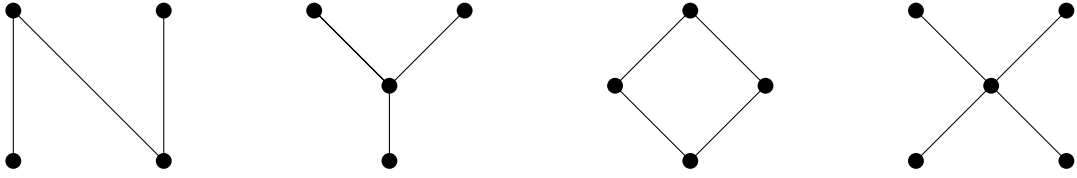

On the other hand, we prove that Conjecture~\ref{conj1} does hold for a class of posets (that does not include $\Diamond$). 
Given~$p\in \mathds N$ and a poset~$P$ on~$[p]$
we define the \emph{dual} $\overline P$ of $P$ as~$\overline P:=\{[p]\setminus F\colon F\in P\}$.
We say a poset~$P$~\emph{has legs} if there are distinct elements~$L_1,L_2,H\in P$ such that~$L_1,L_2$ are incomparable, $L_1,L_2\subseteq H$ and for any other element $A\in P\setminus \{L_1,L_2,H\}$ we have~$A\supseteq H$.
The elements $L_1$ and~$L_2$ are called \emph{legs} and~$H$ is called a \emph{hip}.

\begin{theorem}\label{linearthm}
    Let $P$ be a poset with legs and $n \geq 3$. Then~$\sat^*(n,P)\geq n+1$. 
    Moreover, if both~$P$ and~$\overline P$ have legs, then~$\sat^*(n,P)\geq 2n+2$.
\end{theorem}

Our results still leave both Conjecture~\ref{conj1} and Conjecture~\ref{ivanconj} open, and it is unclear to us which of these conjectures is true. However, if Conjecture~\ref{ivanconj} is true we believe it highly likely that there will be other posets $P$ for which $\text{sat}^*(n, P) = \Theta (\sqrt{n})$. 

\smallskip

It is also natural to seek exact results on $\text{sat}^*(n, P)$. However, 
despite there already  being several papers concerning $\text{sat}^* (n,P)$~\cite{bgjj, dank, fkkm, ivan2, ivan1, klmpp, msw}, there are relatively few posets $P$ for which $\text{sat}^* (n,P)$ is known precisely (see, e.g., Table~1 in~\cite{klmpp} for a summary of most of the known results).
Our next result extends this limited pool of posets, determining $\text{sat}^* (n,X)$ and  $\text{sat}^* (n,Y)$ (see Figure~\ref{fig:posets} for the Hasse diagrams of $X$ and $Y$).

% Despite there already being several papers concerning $\text{sat}^* (n,P)$~\cite{dank, fkkm, ivan2, ivan1, klmpp, msw}, there are very few posets $P$ for which $\text{sat}^* (n,P)$ is known precisely. 
% Our first result extends this limited pool of posets, determining $\text{sat}^* (n,X)$ and  $\text{sat}^* (n,Y)$ (see Figure~\ref{fig:posets} for the definition of $X$ and $Y$).
\begin{theorem}\label{xythm} Given any~$n\in \mathds N$ with~$n\geq 3$, 
\begin{enumerate}[label=\rmlabel]
    \item\label{it:thmy}  $\text{sat}^* (n,Y) = n+2 $ and
    \item\label{it:thmx} $\text{sat}^* (n,X) = 2n+2$. 
\end{enumerate}
\end{theorem}
Note that $\overline X=X$, so Theorem~\ref{xythm}\ref{it:thmx} easily follows via Theorem~\ref{linearthm} and an extremal construction. An application
of Theorem~\ref{linearthm} to $Y$ only yields that $\text{sat}^* (n,Y) \geq n+1 $, so we require an extra idea to obtain 
Theorem~\ref{xythm}\ref{it:thmy}.

\smallskip

In~\cite{msw} a trick was introduced which can be used to prove lower bounds on $\text{sat}^*(n, P) $ for some posets $P$.
The idea is to construct a certain auxiliary digraph $D$ whose vertex set consists of the elements in an induced $P$-saturated family $\mathcal F$; one then argues that how this digraph is defined forces $D$ to contain many edges, which in turn forces a bound on the size of the vertex set of $D$ (i.e., lower bounds $|\mathcal F|)$.
 This trick has been used to prove that $\text{sat}^*(n, \Diamond) \geq \sqrt{n}$~\cite[Theorem~6]{msw} and  $\text{sat}^*(n, N ) \geq \sqrt{n}$~\cite[Proposition~4]{ivan2} (see Figure~\ref{fig:posets} for the Hasse diagram of $N$).

Our proof of Theorem~\ref{mainthm} utilises a variant of this digraph trick. 
In particular, by introducing an appropriate modification to the auxiliary digraph $D$ used in~\cite{msw}, we are able to deduce certain Tur\'an-type properties of $D$. 
Tur\'an problems in digraphs are classical in extremal combinatorics and their study can be traced back to the work of Brown and Harary~\cite{brown}.
Here we prove a Tur\'an-type result concerning \emph{transitive cycles}.

Given~$k\geq 3$, the \emph{transitive cycle on $k$ vertices}~$\diedge{TC}_k$ is a digraph with vertex set~$[k]$ and every directed edge from~$i$ to~$i+1$ for every~$i\in [k-1]$, as well as the directed edge from~$1$ to~$k$. 
We establish an upper bound on the number of edges of a digraph not containing any transitive cycle. 

% Given $k \geq 3$, we write $\overrightharpoonup{TC}_k$ for the \emph{transitive cycle on $k$ vertices}. 
% So $\overrightharpoonup{TC}_k$ has vertex set $[k]$ and there is an edge directed from $i$ to $i+1$ for all $i \in [k-1]$, as well as the directed edge from $1$ to $k$.

\begin{theorem}\label{turanthm}
Let~$n\in\mathds N$ and let~$D$ be a digraph on $n$ vertices. 
If~$D$ is~$\diedge{TC}_k$-free for all~$k\!\geq\! 3$, then
\begin{equation*}\label{eq:turan}
e(D)\leq\max\left\{2(n-1),\left\lfloor\frac{n^2}{4}\right\rfloor\right\}.
\end{equation*}
\end{theorem}
Note that the bound in Theorem~\ref{turanthm} is best possible. Indeed, consider the  $n$-vertex digraph $D$ with vertex classes $A,B$ of size $\left\lfloor n/2\right\rfloor$ and $\left\lceil n/2 \right\rceil$ respectively and all possible directed edges from $A$ to $B$. So $D$ has $\left\lfloor n^2/4\right\rfloor$ edges and contains no transitive cycle. Similarly, pick an arbitrary $n$-vertex tree $T$ and let $D'$ be the digraph with vertex set $V(D')=V(T)$ and with an edge directed from $i$ to $j$ if and only if the unordered pair $ij$ is an edge in $T$. So $D'$ has $2(n-1)$ edges and contains no transitive cycle.

% In particular, we can then use a  result of Brown and Harary~\cite{brown} from the 1960s to deduce the $C=\sqrt{2}$ case of 
% Theorem~\ref{mainthm}. By presenting a bespoke  Tur\'an-type result (Theorem~\ref{turanthm} below), we are able to deduce the $C=2-o(1)$ case. 
In Section~\ref{sec:conc} we explain why exploiting the connection between the induced saturation problem for posets and the Tur\'an problem for digraphs hits a natural barrier if we wish to further improve the lower bound in Theorem~\ref{mainthm}.

\smallskip

We conclude this introductory section by setting out notation that will be used in the rest of the paper.

\subsection*{Notation} Given two elements $A,B$ of a poset $\mathcal{F}\subseteq 2^{[n]}$ we say that \emph{$A$ dominates $B$} if $B\subseteq A$. 
We say $A$ and $B$ are \emph{comparable} if one dominates the other; otherwise we say they are \emph{incomparable}. 
An element $F$ in $\mathcal{F}$ is \emph{maximal} if there is no other element in $\mathcal{F}$ which dominates $F$. 
We define \emph{minimal} analogously.

For a digraph $D=(V(D),E(D))$, we write $\diedge{xy}$ for the edge in $D$ directed from the vertex $x$ to the vertex $y$.
We say that $\diedge{xy}$ is an {\it out-edge} of $x$ and an {\it in-edge} of $y$. For brevity, we write $e(D):=|E(D)|$. Given a set $U\subseteq V(D)$, we denote the induced subgraph of $D$ with vertex set $U$ as $D[U]$.
The \emph{underlying graph of~$D$} is the graph with vertex set~$V(D)$ whose edges are all (unordered) pairs~$\{x,y\}$ such that $\diedge{xy}\in E(D)$ or $\diedge{yx}\in E(D)$. 

Given~$k\geq 2$, an \emph{oriented path} $v_1\dots v_k$ from $v_1$ to $v_k$ consists of the edges $\diedge{v_iv}_{i+1}$ for every $i\in[k-1]$. 
Similarly, an \emph{oriented cycle} $v_1\dots v_k$ consists of the edges $\diedge{v_iv}_{i+1}$ for every $i\in[k-1]$ and $\diedge{v_kv}_{1}$.

\section{The proofs}

\subsection{Proof of Theorem~\ref{turanthm}}
\newparagraph

We proceed by induction on $n$. For the base case, it is easy to check that the statement of the theorem holds for $n\leq3$. Next, we prove the inductive step. 

Let $D$ be a digraph on $n\geq4$ vertices which is $\diedge{TC}_k$-free for all $k\geq3$.

% For~$k\geq 2$, the vertices~$v_1,\dots, v_k$ form an oriented cycle if~$v_{i}v_{i+1}$ is an edge for every $i\in [k]$ where the indices are consider to be modulo~$k$.

% For the following claim we consider a pair of edges~$\diedge{uv}$ and~$\diedge{vu}$ to be an oriented cycle of length two.

\begin{claim}\label{claim:cycle}
If $D$ contains an induced oriented cycle then $e(D)\leq\max\left\{2(n-1),\left\lfloor\frac{n^2}{4}\right\rfloor\right\}$.
\end{claim}
\begin{proofclaim}
Suppose $D$ contains an induced oriented cycle $C$.
For every $v\in V(D)\setminus V(C)$, it is straightforward to check that, since $D[V(C)\cup\{v\}]$ contains no transitive cycle, then
\begin{enumerate}[label=\rmlabel]
    \item there is at most one in-edge of $v$ incident to $V(C)$ and \label{it:in}
    \item there is at most one out-edge of $v$ incident to $V(C)$.\label{it:out}
\end{enumerate}
Let $D'$ be the digraph obtained by contracting the cycle $C$ into one vertex $c$. Namely,~$D'$ has vertex set~$V(D')=( V(D)\cup \{c\})\setminus V(C)$ and  $E(D')$ is the union of the following sets:
\begin{itemize}
    \item $E(D[V(D)\setminus V(C)])$,
    \item $\{\diedge{xc}: \exists \ \diedge{xy}\in E(D),x\notin V(C),y\in V(C)\}$ and
    \item $\{\diedge{cx}: \exists \ \diedge{yx}\in E(D),x\notin V(C),y\in V(C)\}$.
\end{itemize}
Note that  properties \ref{it:in} and \ref{it:out} imply that  $e(D')=e(D)-e(C)=e(D)-|V(C)|$.

Suppose $D'$ contains a transitive cycle $\diedge{TC}_k$ on vertices $v_1,\dots,v_k$ for some $k\geq3$. Namely, $\diedge{v_jv}_{j+1}\in E(D')$ for every $j\in[k-1]$ and $\diedge{v_1v}_k\in E(D')$.
If $c\not=v_i$ for every $i\in[k]$ then $v_1,\dots,v_k$ form a transitive cycle in $D$, a contradiction. 
Therefore, $c=v_i$ for some $i\in[k]$.

For brevity we only consider the case $i\not=1,k$ (the cases $i=1$ and $i=k$ can be handled with a similar argument). 
By the definition of $D'$, there exist $c_1,c_2\in V(C)$ such that \mbox{$\diedge{v_{i-1}c}_1 , \diedge{c_2v}_{i+1}\in E(D)$.} Furthermore, there exists an oriented path $u_1 \dots u_\ell$ from $u_1$ to $u_\ell$ such that $u_1=c_1$, $u_\ell=c_2$ and $u_j\in E(C)$ for all $j\in[\ell]$. Then $v_1\dots v_{i-1} u_1 \dots u_\ell v_{i+1} \dots v_k$ is an oriented path from $v_1$ to $v_k$ in $D$. Together with $\diedge{v_1v}_k\in E(D)$, this forms a transitive cycle in $D$, a contradiction.

Therefore, $D'$ is $\diedge{TC}_k$-free for all $k\geq3$; so by the induction hypothesis we have 
$$e(D')\leq\max\left\{2(n-|V(C)|),\left\lfloor\frac{(n-|V(C)|+1)^2}{4}\right\rfloor\right\}.$$ 
The inequalities $n\ge4$ and $2\leq|V(C)|\leq n$ imply
$$2(n-|V(C)|)+|V(C)|\leq2(n-1)\quad\text{and}\quad\left\lfloor\frac{(n-|V(C)|+1)^2}{4}\right\rfloor+|V(C)|\leq\left\lfloor\frac{n^2}{4}\right\rfloor.$$
In particular, $e(D)=e(D')+|V(C)|\leq\max\left\{2(n-1),\left\lfloor\frac{n^2}{4}\right\rfloor\right\}$. This concludes the proof of the claim.\qedclaim
\end{proofclaim}

Because of Claim~\ref{claim:cycle} we may assume that $D$ contains no double edges (which are induced oriented cycles on two vertices). 
%If $D$ contains an induced oriented cycle then we are done by Claim~\ref{claim:cycle}, so we may assume that $D$ has no double edges. 
Additionally, we may assume that the underlying graph of $D$ contains no triangle,
since such a triangle would either correspond to an induced oriented cycle or a transitive cycle~$\diedge{TC}_3$ in $D$.
By Mantel's theorem, a triangle-free graph on $n$ vertices has at most $\lfloor n^2/4\rfloor$ edges, hence $e(D)\leq\left\lfloor\frac{n^2}{4}\right\rfloor.$
This concludes the inductive step. \qed
%As $D$ has no double edges, we can define $D^{\star}$ to be the simple graph obtained by ignoring the orientation of the edges of $D$. In particular, we have $e(D)=e(D^{\star})$.
%Observe that $D^{\star}$ is triangle-free, hence Mantel's theorem implies
%$$e(D)=e(D^{\star})\leq\left\lfloor\frac{n^2}{4}\right\rfloor.$$

\subsection{Proof of Theorem~\ref{mainthm}}
\newparagraph

% As indicated above, our proof will rely on the following two Tur\'an-type results for digraphs. 

% \begin{thm}\cite{brown}\label{brown}

% \end{thm}

% \begin{thm}\label{turanthm}

% \end{thm}
% SKETCH PROOF of Theorem~\ref{turanthm}

We prove Theorem~\ref{mainthm} using the following two lemmata. 
%Note that a weaker version of  Lemma~\ref{digraph} (where $|\mathcal F|\geq \sqrt{n}$) was implicitly proven in~\cite{msw}.

\begin{lemma}\label{digraph}
Let~$\mathcal{F}\subseteq 2^{[n]}$.
If for every $i\in[n]$ there are elements $A,B\in\mathcal{F}$ such that $A\setminus B=\{i\}$ then $|\mathcal{F}|\geq \min \{ 2 \sqrt{n}, n/2+1 \}$.
\end{lemma}
Notice that Lemma~\ref{digraph} is not specifically about induced saturated families. We will discuss this further in Section~\ref{sec:conc}.

\begin{lemma}\label{blowup}
Given a poset~$P$, let~$\mathcal{F}_0\subseteq 2^{[n_0]}$ be an induced $P$-saturated family.
Suppose there is an $i\in[n_0]$ such that there are no elements $A,B\in\mathcal{F}_0$ satisfying $A\setminus B=\{i\}$.
Then $\text{sat}^*(n,P)\leq|\mathcal{F}_0|$ for every $n\geq n_0$.
\end{lemma}

\begin{proof}[Proof of Theorem~\ref{mainthm}]
Let $P$ be a poset. Suppose that for every $n\in\mathds N$ and every induced $P$-saturated family $\mathcal{F}\subseteq2^{[n]}$ we have that for every $i\in[n]$ there are elements $A,B\in\mathcal{F}$ such that $A\setminus B=\{i\}$. Then Lemma~\ref{digraph} implies $|\mathcal{F}|\geq \min \{ 2 \sqrt{n}, n/2+1 \}$, and thus $\text{sat}^*(n,P)\geq \min \{ 2 \sqrt{n}, n/2+1 \}$.

Otherwise, there exists some $n_0\in\mathds N$ and $i\in[n_0]$ so that there is
 an induced $P$-saturated family $\mathcal{F}_0\subseteq2^{[n_0]}$   such that there are no elements $A,B\in\mathcal{F}_0$ with $A\setminus B=\{i\}$. Then Lemma~\ref{blowup} implies that $\text{sat}^*(n,P)\leq K_P$ for every $n\in\mathds N$ where 
$$K_P: =\max\{|\mathcal{F}_0|, \, \text{sat}^*(m,P):m<n_0\}.$$
\end{proof}

The rest of this subsection covers the proofs of Lemmata~\ref{digraph} and~\ref{blowup}. 
For the proof of Lemma~\ref{digraph} we apply Theorem~\ref{turanthm}.

\begin{proof}[Proof of Lemma~\ref{digraph}]
Let $D$ be a digraph with vertex set $\mathcal{F}$ and edge set $E(D)$ defined as follows: for every $i\in[n]$, 
choose precisely one pair $A,B \in \mathcal F$ with $A\setminus B=\{i\}$; add the edge 
 $\diedge{AB}$ to $E(D)$.
Thus~$D$ has exactly~$n$ edges. Note that the hypothesis of the lemma ensures $D$ is well-defined.

\begin{claim}\label{claim:TC}
For every~$k\geq 3$,~$D$ is~$\diedge{TC}_k$-free.
\end{claim}

\begin{proofclaim}
    It suffices to show that given~$\{A_1, \dots, A_k\}\subseteq \mathcal F$ such that $\diedge{A_jA}_{j+1}\in E(D)$ for every~$j\in [k-1]$ then~$\diedge{A_1A}_k\notin E(D)$.
    By definition of~$D$, there are distinct  $i_1, \dots, i_{k-1}\in [n]$ such that
    \begin{align}\label{eq:diredge}
        A_{j}\setminus A_{j+1} = \{i_j\} \qquad \text{for every }j\in [k-1].
    \end{align}
    This implies $A_{j}\subseteq A_{j+1}\cup\{i_j\}$ for every $j\in [k-1]$, and thus~$A_1\subseteq A_k\cup \{i_1,\dots,i_{k-1}\}$. 
    Hence,
    \begin{align}\label{eq:diredge2}
    A_1\setminus A_{k}\subseteq \{i_1, \dots, i_{k-1}\}.
    \end{align}
    Note that if $\diedge{A_1A}_k\in E(D)$ then \eqref{eq:diredge2} implies that 
     $A_1\setminus A_k=\{i_j\}$ for some $j\in[k-1]$. However, 
    recall that for every~$i\in[n]$ there is exactly one edge~$\diedge{AB}$ in~$D$ such that~$A\setminus B=\{i\}$; so~\eqref{eq:diredge} implies  that $\diedge{A_1A}_k\notin E(D)$, as desired.\qedclaim
\end{proofclaim}

By Claim~\ref{claim:TC}, we can apply Theorem~\ref{turanthm} to the digraph $D$. This yields
$$n = |E(D)|\leq \max \left \{ 2(|\mathcal F|-1), \frac{|\mathcal F|^2}{4} \right \}\,, $$
which implies $|\mathcal{F}|\geq \min \{ 2 \sqrt{n}, n/2+1 \}$. 
\end{proof}

%The intuition behind Lemma~\ref{blowup} is the following. Since there are no %$A,B\in\mathcal{F}$ such that $A\setminus B=\{i\}$, it is easy to check that the %family $\mathcal{F}^*\subseteq 2^{[n]\setminus\{i\}}$ obtained by deleting $i$ %from all elements in $\mathcal{F}$ is induced $P$-saturated. Therefore, we can %think of $\mathcal{F}$ as a family obtained by adding $i$ to some elements of an %induced $P$-saturated family. Since this augmenting process preserved the %induced $P$-saturated property of the family, intuitively we could add

We now present the proof of Lemma~\ref{blowup}. 

\begin{proof}[Proof of Lemma~\ref{blowup}]
Observe that it is enough to prove that for every~$n\!\geq\!n_0$ there exists a family~$\mathcal F\subseteq 2^{[n]}$ such that 
\begin{enumerate}[label=\rmlabel]
    \item $|\mathcal F|=|\mathcal F_0|$, \label{it:sizeofF}
    \item $\mathcal F$ is induced~$P$-saturated and \label{it:Fissat}
    \item there are no elements $A, B\in \mathcal F$ satisfying~$A\setminus B=\{i\}$. \label{it:split}
\end{enumerate}
We proceed by induction on~$n$ and observe that the base case ($n=n_0$) follows directly from the assumption in the statement of the lemma. 

Given an induced $P$-saturated family~$\mathcal F\subseteq 2^{[n]}$ satisfying~\ref{it:sizeofF}-\ref{it:split} we consider the following function $f$ from~$2^{[n]}$ to $2^{[n+1]}$:
$$f(A):=\begin{cases}
A & \text{if $i\not\in A$ and}\\
A\cup \{n+1\} & \text{if $i\in A$.}\end{cases}$$
We shall prove that the family~$\mathcal F' := f(\mathcal F)\subseteq 2^{[n+1]}$ satisfies~\ref{it:sizeofF}-\ref{it:split}.

First, note that~\ref{it:sizeofF} follows directly since~$f$ is injective.
Second, \ref{it:split} also follows easily, since every element $f(A)$ either contains both $i$ and  $n+1$ or neither of them. 
Actually, because of this last property one might say that~$n+1$ behaves as a `copy' of~$i$ in~$f(2^{[n]})$.
We need to prove~\ref{it:Fissat}, i.e., that~$\mathcal F'$ is induced $P$-saturated.

It is easy to check that~$f$ preserves the inclusion/incomparable relations between elements. 
More precisely, for every~$A,B\in 2^{[n]}$ we have 
\begin{align}\label{eq:relations}
    A\subseteq B 
    \quad \Longleftrightarrow\quad
    f(A)\subseteq f(B).
\end{align}
Therefore, if~$P'$ forms an induced copy of~$P$ in~$\mathcal F'$, then~$f^{-1}(P')\subseteq \mathcal F$ forms an induced copy of~$P$ in~$\mathcal F$.
This means that, since~$\mathcal F$ is induced $P$-free, $\mathcal F'$ must be induced~$P$-free as well.
It is left to prove that for any $S\in 2^{[n+1]} \setminus \mathcal{F}'$, there is an induced copy of $P$ in $\mathcal{F'}\cup \{S\}$.
There are four cases to consider depending on~$S$. 
The first two cases are short and  the fourth case is an easy consequence of the third.
Let~$S\in 2^{[n+1]} \setminus \mathcal{F}'$.

\smallskip
\textit{{{\hskip 0.7em First case: $i,n+1\notin S$.}}}
Note that $S\subseteq 2^{[n]}$ and $f(S)=S$, thus $S\not\in\mathcal{F}$ (as otherwise we would have $S\in\mathcal{F}'$). 
As $\mathcal{F}$ is induced $P$-saturated, there exists an induced copy $P'$ of $P$ in $\mathcal{F}\cup \{S\}$. 
By~\eqref{eq:relations} the set $\{f(F):F\in P'\}$ is an induced copy of $P$ in $\mathcal{F'}\cup \{S\}$.

\smallskip
\textit{{{\hskip 0.7em Second case: $i,n+1\in S$.}}}
Set $S^\star:=S\setminus\{n+1\}$. 
Note that $S^\star\subseteq 2^{[n]}$ and $f(S^\star)=S$, thus $S^\star\not\in\mathcal{F}$ (as otherwise we would have $S\in\mathcal{F}'$).
As $\mathcal{F}$ is induced $P$-saturated, there exists an induced copy $P'$ of $P$ in $\mathcal{F}\cup \{S^\star\}$. 
By~\eqref{eq:relations} the set $\{f(F):F\in P'\}$ is an induced copy of $P$ in $\mathcal{F'}\cup {S}$.

\smallskip
\textit{{{\hskip 0.7em Third case: $i \in S$ and~$n+1\notin S$.}}}
For this case we use the assumption that there are no elements $A,B\in\mathcal{F}$ such that $A\setminus B=\{i\}$. 
Let $S^{\star}:=S\setminus\{i\}\in 2^{[n]}$.
We use the sets~$S$ and $S^{\star}$ to find elements~$A,B\in \mathcal F$ that will contradict~\ref{it:split} (for the family~$\mathcal F$).

\begin{claim}\label{claim:A}
Either $\mathcal{F'}\cup \{S\}$ contains an induced copy of $P$ or there exists an element $A\in\mathcal{F}$ such that $A\subseteq S$ and $i\in A$.
\end{claim}
\begin{proofclaim}
Assume that there is no~$A\in \mathcal F$ satisfying the properties of the claim; notice that this implies $S\notin\mathcal F$. We need to prove that $\mathcal F'\cup \{S\}$ contains an induced copy of~$P$.

Since $\mathcal{F}$ is induced $P$-saturated, there exists an induced copy $P'$ of $P$ in $\mathcal{F}\cup \{S\}$ that contains~$S$. 
We shall prove that $\{f(F):F\in P', F\neq S\}\cup \{S\}$ is an induced copy of~$P$ in~$\mathcal F'\cup \{S\}$.

%Our aim is to show that the set $\{f(F):F\in P', F\neq S\}\cup \{S\}$ forms an induced copy of $P$ in $\mathcal{F}'\cup \{S\}$. 
First, observe that \eqref{eq:relations} implies the set~$\{f(F):F\in P',F\not=S\}$ is an induced  copy of the poset~$P'\setminus \{S\}$.
If for every $F\in P'\setminus \{S\}$ the relation (inclusion/incomparability) between $F$ and $S$ is the same as the one between $f(F)$ and $S$ then we are done.
 Thus it is enough to prove for every~$F\in P'\setminus \{S\}$ we have that
\begin{enumerate}[label=\rmlabel]
    \item if $S$ and $F$ are incomparable then $S$ and $f(F)$ are incomparable, \label{it:incompA} 
    \item if $S\subseteq F$ then $S\subseteq f(F)$ and \label{it:subA}
    \item if $S\supseteq F$ then $S\supseteq f(F)$. \label{it:supA}
\end{enumerate}
Notice that~\ref{it:subA} follows directly from~$F\subseteq f(F)$. 
It is easy to check that~\ref{it:incompA} also holds by recalling that~$i\in S$ and~$n+1\notin S$.
So finally, for $F\in P'\setminus \{S\}$ as in~\ref{it:supA}, observe that~$i\notin F$ otherwise $F$ would satisfy the properties of $A$ in the statement of the claim. 
Then $f(F)=F$ and~\ref{it:supA} holds. \qedclaim
\end{proofclaim}

The proof of the following claim is very similar. We include it for completeness.

\begin{claim}\label{claim:B}
Either $\mathcal{F'}\cup \{S\}$ contains an induced copy of $P$ or there exists an element $B\in\mathcal{F}$ such that $S^{\star}\subseteq B$ and $i\not\in B$.
\end{claim}

\begin{proofclaim}
Assume that there is no  $B\in\mathcal{F}$ satisfying the properties of the claim; notice that this implies $S^{\star}\notin\mathcal{F}$. It remains to show that $\mathcal{F}'\cup \{S\}$ contains an induced copy of $P$.

Since $\mathcal{F}$ is induced $P$-saturated, there exists an induced copy $P'$ of $P$ in $\mathcal{F}\cup \{S^{\star}\}$ that contains~$S^{\star}$.
We shall prove that $\{f(F):F\in P', F\not=S^{\star}\}\cup \{S\}$ is an induced copy of $P$ in~$\mathcal F'\cup \{S\}$.

First, observe that \eqref{eq:relations} implies the set~$\{f(F):F\in P',F\not=S^{\star}\}$ is an induced copy of the poset~$P'\setminus \{S^{\star}\}$.
If for every $F\in P'\setminus \{S^{\star}\}$ the relation (inclusion/incomparability) between $F$ and $S^{\star}$ is the same as the one between $f(F)$ and $S$ then we are done.
Thus, for every~$F\in P'\setminus \{S^{\star}\}$ we must prove that
\begin{enumerate}[label=\rmlabel]
    \item if $S^{\star}$ and $F$ are incomparable then $S$ and $f(F)$ are incomparable, \label{it:incompB} 
    \item if $S^{\star}\supseteq F$ then $S\supseteq f(F)$ and \label{it:supB}
    \item if $S^{\star}\subseteq F$ then $S\subseteq f(F)$. \label{it:subB}
\end{enumerate}
For~\ref{it:supB}, notice that $S^{\star}\supseteq F$ implies $i\not\in F$ and thus $S\supseteq S^{\star}\supseteq F=f(F)$. 
It is easy to check that~\ref{it:incompB} also holds by recalling that~$i\in S$ and~$n+1\notin S$.
Finally, for $F\in P'\setminus \{S^{\star}\}$ as in~\ref{it:subB}, observe that~$i\in F$ otherwise $F$ would satisfy the properties of $B$ in the statement of the claim. 
Then $f(F)=F\cup\{n+1\}$ and~\ref{it:subB} holds. \qedclaim
\end{proofclaim}

To finish the proof of this case, suppose that $\mathcal{F'}\cup \{S\}$ does not contain an induced copy of $P$.
By Claims~\ref{claim:A} and~\ref{claim:B} there are elements $A,B\in \mathcal F$ such that $A\subseteq S$, $i\in A$ and $S\setminus \{i\}\subseteq B$, $i\not\in B$. 
In particular, we have $A\setminus B=\{i\}$, a contradiction.

\smallskip
\textit{{{\hskip 0.7em Fourth case: $i\notin S$ and $n+1\in S$.}}}
Let $S':=(S\setminus \{n+1\}) \cup \{i\}$. Recall that every element of $\mathcal F'$ either contains both $i$ and $n+1$ or neither of them. This property ensures that 
\begin{enumerate}[label=\alabel]
\item 
$S' \not \in \mathcal F'$;
\item 
for any $A \in \mathcal F'$, $A \subseteq S$ if and only if $A \subseteq S'$; \label{22}
\item for any $A \in \mathcal F'$, $S \subseteq A$ if and only if $S' \subseteq A$. \label{33}
\end{enumerate}
By the third case we have that $\mathcal F'\cup \{S'\}$ contains an induced copy of $P$. Clearly~\ref{22} and~\ref{33} then imply that 
$\mathcal F'\cup \{S\}$ contains an induced copy of $P$.
\end{proof}

\subsection{Proof of Theorems~\ref{linearthm} and~\ref{xythm}} 
\newparagraph

Given a poset $P$ with legs and an induced~$P$-saturated family $\mathcal F\subseteq 2^{[n]}$, the following lemma already implies that $\mathcal F$ has size at least~$n$.

\begin{lemma}\label{lem:legs}
Let $P$ be a poset with legs and~$\mathcal F\subseteq 2^{[n]}$ be an induced~$P$-saturated family.
Then, there is an injective function~$f\colon [n]\longrightarrow \mathcal F\setminus \{\emptyset\}$ such that 
\[
    f(i) = \begin{cases}
        \{i\}   &\text{if }\{i\}\in \mathcal F \text{ and} \\
        H(i)       &\text{if }\{i\}\notin \mathcal F, \text{ where $H(i)$ is the hip of an induced copy of $P$ in $\mathcal F\cup \{\{i\}\}$}.
    \end{cases}
\]
\end{lemma}

\begin{proof}
%The proof is similar to the one used to prove Lemma 5 in \cite{fkkm}. 
Let $\mathcal{F}$ be an induced $P$-saturated family.
For~$\{i\}\notin \mathcal F$ we shall choose a suitable $f(i)=H(i)$ in such a way that~$f$ is injective.
Given~$i\in [n]$ with~$\{i\}\notin \mathcal F$, let~$P'$ be an induced copy of~$P$ in~$\mathcal{F}\cup\{\{i\}\}$ that contains $\{i\}$.
Observe that~$\{i\}$ can only be a leg of $P'$.
Among all possible choices for $P'$, we pick it under the following conditions.
\begin{enumerate}[label=\rmlabel]
    \item The leg $L'$ of~$P'$ which is not~$\{i\}$ is taken so that $|L'|$ is as large as possible; \label{it:ymax}
    \item under \ref{it:ymax}, the hip $H'$ of $P'$ is taken so that $|H'|$ is as small as possible. \label{it:mmin}
\end{enumerate}
Thus we define $f(i):=H'$. Note that since~$L'$ and~$\{i\}$ are incomparable, $i\notin L'$.

\begin{claim}\label{claim1}
$H'=L'\cup \{i\}$.
\end{claim}

\begin{proofclaim}
Suppose for a contradiction that $L'\cup \{i\}\subsetneq H'$. 
If $L'\cup \{i\}\in \mathcal{F}$ then the poset~$(P'\cup \{L'\cup \{i\}\}) \setminus \{H'\}$ is an induced copy of~$P$, which contradicts~\ref{it:mmin}.
Therefore $L'\cup \{i\}$ is not in $\mathcal{F}$ and so $\mathcal{F}\cup \{ L'\cup \{i\} \}$ contains an induced copy $P''$ of  $P$ which uses the set $L'\cup \{i\}$.

Denote the legs of~$P''$ by~$L_1'',L_2''$.
If $L'\cup \{i\}$ is not a leg in $P''$ then~$L_1'', L_2''\subset L'\cup \{i\}$.
This implies that $(P'\cup \{L_1'', L_2''\}) \setminus\{\{i\}, L'\}$ is an induced copy of  $P$ in $\mathcal{F}$, a contradiction. 
Thus, $L'\cup \{i\}$ is a leg, and we may assume $L_1''=L'\cup \{i\}$. 

Note that if $L'$ and $L_2''$ are incomparable, then $(P''\cup \{L'\})\setminus \{L_1''\}$ would be an induced copy of $P$ in~$\mathcal F$, a contradiction.
Thus, $L'$ is comparable with~$L_2''$.
In particular~$L'\subset  L_2''$, otherwise we would have that $L_2'' \subseteq L'\cup \{i\}=L_1''$, which contradicts the fact that~$L_1''$ and~$L_2''$ are incomparable.
Moreover, again because~$L_1''$ and~$L_2''$ are incomparable, we have that~$i\notin L_2''$.

Finally, observe that~$(P''\cup \{\{i\}\})\setminus \{L_1''\}$ is an induced copy of $P$ in~$\mathcal F \cup \{\{i\}\}$ with legs~$\{i\}$ and~$L_2''\supset L'$, which contradicts~\ref{it:ymax}.\qedclaim
\end{proofclaim}

Recall  $f(i):=H'= H(i)$.
We shall prove that $f$ is injective. Suppose not, namely $f(i)=f(j)$ for some $i\not=j$.
If either $\{i\}\in\mathcal{F}$ or $\{j\}\in\mathcal{F}$ then we get a contradiction.
Thus, we have~$\{i\}, \{j\}\notin \mathcal F$, and hence there exist induced copies $P'\subseteq \mathcal{F}\cup\{\{i\}\}$ and $P''\subseteq \mathcal{F}\cup\{\{j\}\}$  of $P$ such that $H$ is the hip of both and~$f(i)=f(j)=H$.
Say $\{i\}$ and $L'$ are the legs of $P'$ while $\{j\}$ and $L''$ are the legs of $P''$. 
Because of Claim~\ref{claim1} we have that~$H=L'\cup \{i\}=L''\cup \{j\}$, and hence~$L'$ and~$L''$ are incomparable.
This implies that $(P'\cup\{ L''\})\setminus \{\{i\}\}$ is an induced copy of $P$ in $\mathcal{F}$, a contradiction.
\end{proof}

Now we are ready to prove Theorem~\ref{linearthm}.

\begin{proof}[Proof of Theorem~\ref{linearthm}]
Observe that for a poset $P$ with legs, the inequality~$\sat^*(n,P)\geq n+1$ follows directly by applying Lemma~\ref{lem:legs} and by noticing that the empty set is in every induced~$P$-saturated family.

For $P$ such that $P$ and~$\overline P$ have legs, we proceed as follows.
Let~$\mathcal F \subseteq 2^{[n]}$ be an induced $P$-saturated family and observe that~$\emptyset,[n]\in\mathcal F$.
The poset~$\overline {\mathcal F} \subseteq 2^{[n]}$ is induced $\overline P$-saturated and therefore we can apply Lemma~\ref{lem:legs} to both~$\mathcal F$ and $\overline{\mathcal F}$ to obtain injective functions~$f$ and $\overline f$ respectively. 
Since both $P$ and $\overline P$ have legs, it is easy to see that the hip of $P$ is not a maximal element of $P$.
It follows that $f(i)\neq [n]$ for every~$i\in [n]$. 
That is, $f([n]) \subseteq \mathcal F\setminus \{\emptyset, [n]\}$; similarly $\overline f([n]) \subseteq \overline{\mathcal F}\setminus \{\emptyset, [n]\}$.

By taking~$[n]\setminus\overline f(i)$ for every~$i\in [n]$ we define a new injective function~$\tilde f$ from $[n]$ to~$\mathcal F\setminus \{\emptyset, [n]\}$:
\[
    \tilde f(i) := \begin{cases} 
    [n]\setminus\{i\}  &\text{if }\{i\}\in \overline{\mathcal F}, \text{ and} \\
    H(i) &\text{if }\{i\}\notin \overline{\mathcal F}, 
    \text{ where $[n]\setminus H(i)$ is the hip of an induced~$\overline P$ in~$\overline {\mathcal F}\cup \{\{i\}\}$}\,.
    \end{cases}
\]

% \[
%     \tilde f(i) := \begin{cases} 
%     [n]\setminus\{i\}  &\text{if }[n]\setminus\{i\}\in \mathcal F, \text{ and} \\
%     H(i) &\text{if }[n]\setminus\{i\}\notin \mathcal F, 
%     \text{ where $[n]\setminus H(i)$ is the hip of an induced copy of~$\overline P$ in~$\overline {\mathcal F}\cup \{\{i\}\}$}\,.
%     \end{cases}
% \]

Now, since~$f$ and~$\tilde f$ are injective, proving that~$f([n])\cap \tilde f([n])=\emptyset$ yields that $\vert \mathcal F\vert \geq 2n+2$ by recalling that the sets $\emptyset$ and~$[n]$ belong to~$\mathcal F$.

Suppose~$f(i)=\tilde f(j)$ for~$i,j\in [n]$. 
If~$\{i\}\in \mathcal F$ or~$[n]\setminus\{j\}\in \mathcal F$ then it is easy to check that such an identity is not possible. 
Hence, we have that~$f(i)=\tilde f(j)=H$ where: $H$ is the hip of an induced copy $P_1$  of $P$ in $\mathcal F\cup \{\{i\}\}$ and $[n] \setminus H$ is the hip of an induced copy $\overline P_2$ of $\overline P$  in $\overline { \mathcal F }\cup \{\{j\}\}$. In particular, this means
that there is an induced copy $P_2$ of $P$ in 
 $ { \mathcal F }\cup \{[n]\setminus \{j\}\}$ where $[n]\setminus \{j\}$ plays the role in $P_2$ of one of the two maximal elements of $P$; $H$ is in $P_2$ and dominates all elements of $P_2$ except for the two maximal elements.

Let $L^1$ be the leg of $P_1$ other than $\{i\}$.
Let~$L_1^2,L_2^2\in \mathcal F$ be the legs of~$P_2$; in particular, $L_1^2$ and~$L_2^2$ are incomparable and $L_1^2, L_2^2\subset H$. This implies that 
 for every~$A\in P_1\setminus \{L^1, \{i\}\}$, $L_1^2, L_2^2 \subset A $.
Thus, $(P_1\cup \{L_1^2, L_2^2\})\setminus  \{L^1, \{i\}\}$ is an induced copy of~$P$ in~$\mathcal F$, a contradiction.
\end{proof}

Theorem~\ref{xythm} now follows by a simple application of Lemma~\ref{lem:legs} and Theorem~\ref{linearthm}, together with two upper bound constructions.

\begin{proof}[Proof of Theorem~\ref{xythm}]

Observe that~$\sat^*(n,Y)=\sat^*(n,\inv{Y})$ where $\inv{Y}:=\overline{Y}$. 
Let~$\mathcal F\subseteq 2^{[n]}$ be an induced $\inv{Y}$-saturated family. 
Since~$\inv{Y}$ has legs, Lemma~\ref{lem:legs} yields an injective function~$f$ from $[n]$ to~$\mathcal F\setminus \{\emptyset\}$ with $f(i) =\{i\}$ if $\{i\}\in \mathcal F$ and $f(i)=H$ otherwise, where $H$ is the hip of an induced copy of $\inv{Y}$ in $\mathcal F\cup \{\{i\}\}$. 
This already implies $|\mathcal F|\geq n+1$ as $\emptyset \in \mathcal F$.

Observe that~$f(i)\neq [n]$ for every~$i\in [n]$, therefore, if~$[n]\in \mathcal F$ then $|\mathcal F|\geq n+2$, as desired.
Hence, assume that $[n]\notin \mathcal F$, which means that~$\mathcal F\cup\{[n]\}$ contains an induced copy of~$\inv{Y}$ that uses~$[n]$. 
Let $L_1,L_2$ be the legs of this copy of $\inv{Y}$ and $H$ be the hip.
Assume there is an $i\in [n]$ such that~$f(i)=H$; so $H$ is the hip of an induced copy of~$\inv{Y}$ contained in~$\mathcal F \cup \{\{i\}\}$. 
Let~$M$ be the maximal element of this copy of $\inv{Y}$ and observe that $\{L_1,L_2,H,M\}$ forms an induced copy of~$\inv{Y}$ in~$\mathcal F$, a contradiction.
Hence, $f(i)\neq H$ for every $i\in [n]$, which means that~$H$ was not counted before and therefore~$\vert \mathcal F\vert \geq n+2$. This argument implies that $\sat^*(n,Y)\geq n+2$.

For the poset~$X$ observe that~$X$ and $\overline X=X$ have legs, therefore Theorem~\ref{linearthm} directly implies that~$\sat^*(n,X)\geq 2n+2$. 
For the upper bounds, consider the posets 
$$P:=\{F\in 2^{[n]}\colon \vert F\vert\geq n-1 \text{ or } F=\emptyset\} \quad \text{and}\quad Q:=\{F\in 2^{[n]}\colon \vert F\vert\geq n-1 \text{ or } \vert F\vert \leq 1\}\,,$$
and notice that they are respectively induced $Y$-saturated and induced $X$-saturated. 
Furthermore, $\vert P\vert= n+2$ and~$\vert Q\vert = 2n+2$. 
\end{proof}

We conclude this subsection by exhibiting a class of posets for which the bound in Theorem~\ref{linearthm} is sharp up to an additive constant.

For any $\ell\in\mathds N$, let $\wedge_\ell$ denote the poset with the following properties:
\begin{itemize}
    \item[-] $\wedge_\ell$ has legs $L_1,L_2$ and hip $H_1$;
    \item[-] $\wedge_\ell\setminus\{L_1,L_2\}=\{H_1,\dots,H_\ell\}$ where $H_j\subset H_{j+1}$ for every $j\in[\ell-1]$. 
\end{itemize}
Similarly, for any $\ell\in\mathds N$, let $X_\ell$ denote the poset with the following properties:
\begin{itemize}
    \item[-] $X_\ell$ has legs $L_1',L_2'$ and 
     $X_\ell\setminus\{L_1',L_2'\}=\vee_\ell$ where  $\vee_\ell :=\overline{\wedge_\ell}$.
\end{itemize}
Clearly $X_1=X$ and~$\wedge_2=\inv{Y}$.
Moreover, $\wedge_\ell, X_\ell$ and $\overline X_\ell$ have legs for every~$\ell\in \mathds N$, and therefore Theorem~\ref{linearthm} implies that $\text{sat}^*(n,\wedge_\ell)\geq n+1$ and $\text{sat}^*(n,X_\ell)\geq 2n+2$. 
The next proposition states that these bounds are close to the exact values of $\text{sat}^*(n,\wedge_\ell)$ and $\text{sat}^*(n,X_\ell)$.

\begin{proposition}\label{upperbounds}
For all integers $n-1>\ell\geq 2$ we have,\footnote{Note that the case $\ell=1$ is covered by Theorem~\ref{xythm} since $\wedge_2=\inv{Y}$ and $X_1=X$.}
\begin{alignat*}{3}
    n+1 &\leq\text{sat}^*(n,\wedge_{\ell+1}) &&\leq n+2^{\ell+1}-\ell-1 \text{ and}  \\
    2n+2 &\leq\text{sat}^*(n,X_\ell)        &&\leq 2n+2^{\ell+1}-2\ell.
\end{alignat*}
\end{proposition}
\begin{proof}
First, we consider $\wedge_{\ell+1}$. Let $\mathcal{F}\subseteq 2^{[n]}$ be the family containing precisely the following sets:
\begin{itemize}
    \item[-] the empty set $\emptyset$;
    \item[-] $\{i\}$ for every~$i\in [n]$;
    \item[-] all subsets of $\{1,2,\dots,\ell\}$;
    \item[-] all proper supersets of $[n]\setminus\{1,2,\dots,\ell\}$.
\end{itemize}
It is straightforward to check that $\mathcal{F}$ has $n+2^{\ell+1}-\ell-1$ elements and  is induced $\wedge_{\ell+1}$-saturated. 
Similarly, the family $\mathcal F':=\mathcal{F}\cup\overline{\mathcal{F}}\subseteq 2^{[n]}$ has $2n+2^{\ell+1}-2\ell$ elements and  is induced $X_\ell$-saturated. 
\end{proof}

\section{Concluding remarks}\label{sec:conc}

The following example shows that the bound in Lemma~\ref{digraph} is essentially tight.

\begin{example}\label{ex:barrier}
Let $n\in\mathds N$ be a perfect square, i.e., $\sqrt{n}\in\mathds N$. Let $A_s, B_t\subseteq[n]$ be defined as follows.
\begin{itemize}
   \item $A_s:=\{s\sqrt{n}+1,s\sqrt{n}+2,\dots,s\sqrt{n}+\sqrt{n}\}$ for every $s\in[\sqrt{n}-1]\cup\{0\}$;
   \item $B_t:=[n]\setminus\{t,t+\sqrt{n},t+2\sqrt{n},\dots,t+(\sqrt{n}-1)\sqrt{n}\}$ for every $t\in[\sqrt{n}]$.
\end{itemize}
Let $\mathcal{F}^{\star}:=\{A_s:s\in[\sqrt{n}-1]\cup\{0\}\}\cup\{B_t:t\in[\sqrt{n}]\}\subseteq 2^{[n]}$. 

Observe that $\mathcal{F}^{\star}$ has $2\sqrt{n}$ elements. Furthermore, for every $i\in[n]$, there exists exactly one pair of elements $A,B\in\mathcal F^{\star}$ such that $A\setminus B=\{i\}$. Namely, if $i=s\sqrt{n}+t$ where $s\in[\sqrt{n}-1]\cup\{0\}$ and $t\in[\sqrt{n}]$ then $A_s\setminus B_t=\{i\}$.
\end{example}

Note that the digraph $D$ with $V(D)=\mathcal F^{\star}$   and edge set~$E(D)=\{\diedge{AB}:A\setminus B=\{i\},i\in[n]\}$ is precisely the balanced oriented bipartite graph (i.e., an extremal example for Theorem~\ref{turanthm}).

This example shows that  if one can improve the lower bound in Theorem~\ref{mainthm} by using the auxiliary digraph approach, then one will really need to use the fact that the digraph is generated by an induced $P$-saturated family (recall this was not assumed in the statement of Lemma~\ref{digraph}). On the other hand, if Theorem~\ref{mainthm} is close to being best possible, then Example~\ref{ex:barrier} points in the direction of potential extremal examples. That is, is there some poset $P$ such that there is a minimum induced $P$-saturated family $\mathcal F \subseteq 2^{[n]}$ that is `close' to the family $\mathcal F^{\star}$ in Example~\ref{ex:barrier}?

\smallskip

Another natural question is to characterise those posets $P$ for which $\text{sat}^*(n,P)$ is bounded by a constant. Observe that Lemma~\ref{blowup} provides a method for determining such posets. That is, if one can exhibit an $n_0 \in \mathbb N$ and an induced $P$-saturated family $\mathcal F\subseteq2^{[n_0]}$ such that for some $i\in[n_0]$ there are no elements $A,B\in\mathcal F$ with $A\setminus B=\{i\}$, then $\text{sat}^*(n,P)=O(1)$.

Along the lines of this research direction, Keszegh, Lemons, Martin, P\'alv\"olgyi and Patk\'os \cite{klmpp} conjectured the following. Given a poset $P$ on $[p]$, let $\dot P$ denote the poset on $[p+1]$ where $\dot P:=P\cup\{[p+1]\}$ (i.e., $\dot P$ is obtained by adding an element to $P$ which dominates all elements in $P$).
\begin{conjecture}\label{conj:dot}\cite{klmpp}
$\text{sat}^*(n,P)=O(1)$ if and only if $\text{sat}^* (n,\dot P)=O(1)$.
\end{conjecture}
Note that neither direction of Conjecture~\ref{conj:dot} has been verified except for some special cases (see \cite[Theorem~3.6]{klmpp}).
It would be interesting to investigate if Lemma~\ref{blowup} can help tackle Conjecture~\ref{conj:dot}.

\smallskip

Finally, it is natural to consider induced saturation problems for families of posets. Given a family of posets $\mathcal P$, we say that $\mathcal F\subseteq 2^{[n]}$ is \emph{induced $\mathcal P$-saturated} if $\mathcal F$ contains no induced copy of any poset $P\in\mathcal P$ and for every $S\in 2^{[n]}\setminus \mathcal F$ there exists an induced copy of some poset $P\in\mathcal P$ in $\mathcal F\cup\{S\}$. We denote the size of the smallest such family by $\text{sat}^*(n,\mathcal P)$. By following the proof of Theorem~\ref{mainthm} precisely, one obtains the following result.

\begin{theorem}\label{LBfamilies}
For any family of posets $\mathcal P$, either there exists a constant $K_{\mathcal P}$ with $\text{sat}^*(n, \mathcal P) \leq K_{\mathcal P}$ or $\text{sat}^*(n, \mathcal P) \geq \min \{ 2 \sqrt{n}, n/2+1 \}$, for all $n \in \mathds N$. 
\end{theorem}
In light of Theorem~\ref{LBfamilies} it is natural to ask whether an analogue of Conjecture~\ref{conj1} is true in this more general setting, or whether (for example) the lower bound on $\text{sat}^*(n, \mathcal P)$ in Theorem~\ref{LBfamilies} is best possible up to a multiplicative  constant. 

%%%%%%%%%%%%%%%%%%%%%%%%%%%%%%%%%%%%%%%%%%%%%%%%%%%
%%%%%%%%%%%%%%%%%%%%%%%%%%%%%%%%%%%%%%%%%%%%%%%%%%%

% ACKNOWLEDGEMENTS
% Include acknowledgements to colleagues and referee here.
% Funding and grant support should appear in footnotes on the front page, using the 
% thanks command in the authors command (see above).

\section*{Acknowledgements}

The authors are grateful to the referees for their helpful and careful reviews. We are also grateful to Louis DeBiasio for pointing out a modification of our original proof of Theorem~\ref{turanthm} which prevents the appearance of a constant error term in the statement of the theorem.

\medskip

{\noindent \bf Data availability statement.}
There are no additional data beyond that contained within the main manuscript.

\end{document}